\documentclass[a4paper]{amsart}
\usepackage{amsmath,amssymb,amsthm,mathtools}
\usepackage[margin=40mm]{geometry}
\usepackage{enumitem}
\usepackage{hyperref}
\hypersetup{pdftitle={Finite Chains of Two-Complexes and Acyclic Covers},
  pdfauthor={Laurent Bartholdi and Roman Mikhailov}}
\setlist{topsep=3pt,itemsep=2pt}
\setlist[enumerate]{leftmargin=10mm}
\setlist[itemize]{leftmargin=8mm}

\theoremstyle{plain}
\newtheorem*{maintheorem}{Theorem A}
\newtheorem*{maincorollary}{Corollary B}
\newtheorem{theorem}{Theorem}[section]
\newtheorem{proposition}[theorem]{Proposition}
\newtheorem{lemma}[theorem]{Lemma}

\theoremstyle{definition}
\newtheorem{definition}[theorem]{Definition}
\theoremstyle{remark}
\newtheorem{remark}[theorem]{Remark}

\numberwithin{equation}{section}

\newcommand{\Z}{\mathbb Z}
\newcommand{\Fp}{\mathbb F_p}
\newcommand{\pres}[1]{\langle #1\rangle}
\newcommand{\cx}[1]{K(#1)}
\newcommand{\Cone}{\operatorname{Cone}}
\newcommand{\supp}{\operatorname{supp}}
\newcommand{\Tp}{T^{+}}

\title{Finite Chains of Two-Complexes and Acyclic Covers}
\author{Laurent Bartholdi}
\author{Roman Mikhailov}
\date{October 3, 2026}

\begin{document}
\begin{abstract}
A negative answer to Whitehead's asphericity problem would follow from an infinite
chain of two-complexes, starting with a non-aspherical one, in which every inclusion
induces the zero map on second homotopy groups. We prove that a connected
two-complex $K$ is the first term of such chains of every \emph{finite} length if
and only if $K$ has a connected acyclic regular cover. In particular, there is a
finite non-aspherical two-complex $K$, a presentation complex of $\mathrm{SL}(2,5)$,
such that for every $n$ there is a strictly increasing chain of finite two-complexes
$K=K_0\subset K_1\subset\dots\subset K_n$ in which all inclusions are zero on $\pi_2$.
\end{abstract}
\maketitle

\section{Introduction}\label{sec:intro}

A two-complex is \emph{aspherical} if its second homotopy group vanishes. Whitehead
asked whether every subcomplex of an aspherical two-complex is aspherical
\cite{Whitehead}; see \cite{Bogley} for a survey. Suppose that
\begin{equation}\label{eq:infinite-chain}
K_0\subset K_1\subset\cdots,\quad \pi_2(K_0)\ne0,\quad
\bigl(\pi_2(K_i)\to\pi_2(K_{i+1})\bigr)=0\ \forall i,
\end{equation}
is a chain of two-complexes. Every map $S^2\to\bigcup_iK_i$ has compact image, hence
lands in some $K_i$, and is therefore null-homotopic in $K_{i+1}$. So
$L=\bigcup_iK_i$ is aspherical and contains the non-aspherical subcomplex $K_0$, a
negative answer to Whitehead's question \cite[\S1]{Hitchman}. Conversely, Howie
\cite{Howie} showed that if the answer is negative, there is a counterexample that
either is finite or is the union of an infinite chain of finite non-aspherical
subcomplexes as in~\eqref{eq:infinite-chain}, with every inclusion null-homotopic.

This paper produces examples of two-complexes that admit a chain as in~\eqref{eq:infinite-chain} of arbitrary finite length, via the following characterization:
\begin{maintheorem}\label{thm:A}
For a two-complex $K$ the following are equivalent.
\begin{enumerate}[label=\textup{(\arabic*)}]
\item For every $n\ge1$ there is a chain of two-complexes
$K=K_0\subsetneq K_1\subsetneq\cdots\subsetneq K_n$
in which every inclusion induces the zero map on $\pi_2$.
\item $K$ has a connected acyclic regular cover.
\end{enumerate}
If $K$ is finite and acyclic, the $K_i$ in~\textup{(1)} can be chosen finite.
\end{maintheorem}

Recall that a two-complex $K$ is \emph{Cockcroft} \cite{Cockcroft} if its Hurewicz
map $h_K\colon\pi_2(K)\to H_2(K)$ is zero. By Hopf's exact sequence
$\pi_2(K)\to H_2(K)\to H_2(\pi_1K)\to0$ \cite[Ch.~II, \S5]{Brown}, this holds if and
only if $H_2(K)\to H_2(\pi_1K)$ is an isomorphism. Moreover $K$ is Cockcroft if and only
if there is a single strict inclusion $K\subsetneq L$ of two-complexes that is zero
on~$\pi_2$ (Lemma~\ref{lem:toolbox}(a)). In particular, in a chain as in~(1) the
complexes $K_0,\dots,K_{n-1}$ are Cockcroft, and truncating a chain of length $n+1$
shows that $K_n$ can be taken Cockcroft too. Theorem~A thus describes the complexes
that are ``Cockcroft to every finite order''.

Theorem~A does not produce an infinite chain~\eqref{eq:infinite-chain}: the chains
constructed for different $n$ are not extensions of one another.

\begin{maincorollary}
Let $K$ be the presentation complex of
\[A=\pres{x,y\mid x^2yx^{-1}y,\ xy^4xy^{-1}}.\]
Then $\pi_2(K)\cong\Z^{119}$, and for every $n\ge1$ there is a chain of finite
two-complexes $K=K_0\subsetneq K_1\subsetneq\cdots\subsetneq K_n$ in which every
inclusion is zero on $\pi_2$.
\end{maincorollary}
\begin{proof}
$K$ has one vertex, two edges and two $2$-cells, so its cellular boundary $\partial_1$
vanishes and $\partial_2$ is the exponent-sum matrix
$\begin{psmallmatrix}1&2\\2&3\end{psmallmatrix}$, of determinant $-1$. Hence $K$ is
acyclic, so it is its own connected acyclic regular cover, and Theorem~A gives the
chains, with finite terms. A coset enumeration shows that $G=\pi_1(K)$ has order
$120$; being perfect, $G$ is $\mathrm{SL}(2,5)$. The universal cover $\widetilde K$ is
simply connected with Euler characteristic $120\,\chi(K)=120$, so
$\pi_2(K)\cong H_2(\widetilde K)\cong\Z^{119}$.
\end{proof}

\subsection*{Previous work}
The second author \cite[Theorem~1.1]{Mikhailov} constructed chains
$K_0\subset K_1\subset K_2$ of finite presentation complexes with $\pi_2(K_0)\ne0$ and
both inclusions zero on $\pi_2$. Corollary~B provides such chains of every finite length.

\subsection*{Outline of the proof}
\begin{itemize}
\item \textbf{(1)$\Rightarrow$(2)} (Section~\ref{sec:necessity}). A chain of length $m$
gives $m+1$ quotients of $G=\pi_1(K)$. To every $2$-chain $v$ on the universal cover of
$K$ we associate a ``cycle test'', and show that $v$ fails the test at no more than one
of these quotients. Hence any $m$ tests are passed simultaneously by some quotient.
A compactness argument in the space of normal subgroups of $G$ then gives a normal
subgroup $N$ passing all tests, which means $H_2(K_N)=0$ with all coefficients. If $N$
is minimal with this property, a Strebel-type lemma shows that $N$ is perfect, so the
cover $K_N$ is acyclic.
\item \textbf{(2)$\Rightarrow$(1)} (Sections~\ref{sec:descent}--\ref{sec:blocks}).
In Section~\ref{sec:descent} we show that it suffices to build the chain over the
acyclic cover $D$ itself, with $\pi_1(D)$ dying after the first step; the chain
descends to $K$ by a pushout. In Section~\ref{sec:acyclic} we build such chains over an
acyclic $D$, by iterating an operation that turns generators into commutators and
replaces relators by ``blocks'', and then appending a terminal extension. The blocks
are used through three properties \ref{B1}--\ref{B3}. In Section~\ref{sec:blocks} we
construct the blocks from a truncated non-positively curved cube complex, using
constructions of Davis \cite{Davis} and of Davis--Januszkiewicz--Weinberger \cite{DJW}.
\end{itemize}

\subsection*{Conventions}
A \emph{two-complex} is a connected CW-complex of dimension at most~$2$, possibly with
infinitely many cells. An inclusion $K\subsetneq L$ means that $K$ is a proper
subcomplex of $L$. Homology has integer coefficients unless stated otherwise. A
connected space is \emph{acyclic} if its reduced integral homology vanishes. For a
connected complex $K$ with $G=\pi_1(K)$ and a normal subgroup $N\trianglelefteq G$, we
write $K_N$ for the connected regular cover with $\pi_1(K_N)=N$; its deck group is $G/N$.
We write $\widetilde X$ for the universal cover of $X$, $h_X\colon\pi_2(X)\to H_2(X)$
for the Hurewicz map, and we identify $\pi_2(X)=H_2(\widetilde X)$ as
$\Z[\pi_1X]$-modules (Hurewicz). A map is \emph{zero on $\pi_2$} if it induces the zero
map on second homotopy groups.

\subsection*{AI usage and formalization}
A variety of AI tools, including ChatGPT 6 Astra, Claude 5.5 Fable, and Aristotle, were used to assist with the proof of the main results and the redaction of the manuscript. The interaction with the AI involved a large amount of back-and-forth and experimentation during approximately 5 months. The arguments and text have been reviewed, checked, and substantially revised by the authors, who are responsible for the final manuscript.

Theorem~\ref{thm:A} has been formalized in \textsf{Lean/Mathlib} and will shortly be available on the Palomar repository.

\section{Preliminaries}\label{sec:prelim}

\subsection{Cockcroft complexes and zero maps}
\begin{lemma}\label{lem:toolbox}
Let $A,X,Y,K,L$ be connected CW-complexes.
\begin{enumerate}[label=\textup{(\alph*)}]
\item Let $K$ be a two-complex. There is a two-complex $L\supsetneq K$ with
$K\subset L$ zero on $\pi_2$ if and only if $K$ is Cockcroft. If $K$ is finite, $L$ can
be chosen finite.
\item If $A$ is Cockcroft and $f\colon A\to Y$ is trivial on $\pi_1$, then $f$ is zero
on $\pi_2$.
\item Let $f\colon A\to X$ be such that $\pi_2(X)=\Z[\pi_1X]\cdot f_*\pi_2(A)$. If $A$ is
Cockcroft, so is $X$. If $g\colon X\to Y$ satisfies $g_*f_*=0$ on $\pi_2(A)$, then $g$ is
zero on $\pi_2$.
\item If $g\colon X\to Y$ is injective on $H_2$ and $Y$ is Cockcroft (e.g.\ $Y$ is
aspherical), then $X$ is Cockcroft. A wedge of Cockcroft complexes is Cockcroft.
\end{enumerate}
\end{lemma}
\begin{proof}
(a) Let $i\colon K\subset L$ with $L$ a two-complex. Since there are no $3$-cells,
$H_2=\ker\partial_2$ on cellular chains, so $H_2(i)$ is injective. If $i$ is zero on
$\pi_2$, then $H_2(i)\circ h_K=h_L\circ i_*=0$, hence $h_K=0$. Conversely, let $K$ be
Cockcroft. If $\pi_1(K)\ne1$, attach $2$-cells along loops representing a generating
set of $\pi_1(K)$ (finite if $K$ is finite). The resulting $L$ is simply connected, so
$h_L$ is an isomorphism, and $h_L\circ i_*=H_2(i)\circ h_K=0$ gives $i_*=0$. If
$\pi_1(K)=1$, then $h_K$ is an isomorphism and zero, so $\pi_2(K)=0$, and we may take
$L=K\cup e^1\cup e^2$ with $e^2$ attached along the loop $e^1$.

(b) By the lifting criterion, $f$ lifts to $\tilde f\colon A\to\widetilde Y$, and
$f_*=p_*\tilde f_*$ where $p\colon\widetilde Y\to Y$ is the covering. Now
$h_{\widetilde Y}\tilde f_*=H_2(\tilde f)h_A=0$, and $h_{\widetilde Y}$ is an
isomorphism, so $\tilde f_*=0$.

(c) The Hurewicz map is invariant under the action of $\pi_1$, so
$h_X(\gamma\cdot f_*s)=h_X(f_*s)=H_2(f)h_A(s)=0$ for $\gamma\in\pi_1X$, $s\in\pi_2A$.
Similarly $g_*(\gamma\cdot f_*s)=g_*(\gamma)\cdot g_*f_*s=0$.

(d) We have $H_2(g)\circ h_X=h_Y\circ g_*=0$ and $H_2(g)$ is injective, so $h_X=0$.
For a wedge $W=\bigvee_rX_r$ with retractions $\rho_r\colon W\to X_r$, the map
$H_2(W)\to\prod_rH_2(X_r)$ with coordinates $H_2(\rho_r)$ is injective, and
$H_2(\rho_r)h_W=h_{X_r}\rho_{r*}=0$.
\end{proof}

\subsection{Regular covers}\label{sec:fox}
Let $K$ be a two-complex and $G=\pi_1(K)$. Write $I$ and $J$ for the sets of $1$-cells
and $2$-cells of $K$, and fix an orientation and a lift to $\widetilde K$ of every cell.
For $N\trianglelefteq G$, the images of these lifts in $K_N$ are lifts of the cells of
$K$, and the cells of $K_N$ are their translates under the deck group $G/N$. Hence, for
$k=\Z$ or $k=\Fp$, the cellular chain module $C_2(K_N;k)$ is identified with
$k[G/N]^{(J)}$, and $C_1(K_N;k)$ with $k[G/N]^{(I)}$ (direct sums, left modules). Let
\[
\partial_{N}\colon k[G/N]^{(J)}\to k[G/N]^{(I)}
\]
be the cellular boundary map in these bases. For $v\in k[G]^{(J)}$ we write
$v_N\in k[G/N]^{(J)}$ for its image; then $\partial_N(v_N)=(\partial_1v)_N$, since the
covering $\widetilde K\to K_N$ is cellular and maps chosen lifts to chosen lifts.

\begin{proposition}\label{prop:fox}
$H_2(K_N;k)=\ker\partial_N$ \textup{(}over $k[G/N]$\textup{)}, and
$H_1(K_N)=N/[N,N]$.
\end{proposition}
\begin{proof}
$K_N$ has no $3$-cells, so $H_2=\ker\partial_2$; and
$H_1(K_N)=\pi_1(K_N)^{\mathrm{ab}}=N/[N,N]$.
\end{proof}

\begin{remark}\label{rem:fox}
If $K$ has a single vertex, it is the presentation complex of a presentation
$\pres{x_i\ (i\in I)\mid r_j\ (j\in J)}$ of $G$, and $\partial_N$ is right
multiplication by the matrix $\bigl(\overline{\partial r_j/\partial x_i}\bigr)$ of Fox
derivatives, reduced to $\Z[G/N]$ \cite{Fox}, \cite[Ch.~II, \S3]{LyndonSchupp}. For
$N=G$ this is the exponent-sum matrix. Every two-complex is homotopy equivalent to such
a complex, by collapsing a maximal tree of its $1$-skeleton.
\end{remark}

\begin{lemma}\label{lem:components}
Let $K\subset X$ be connected CW-complexes and let $M$ be the kernel of
$\pi_1(K)\to\pi_1(X)$. Then every component of the preimage of $K$ in $\widetilde X$ is
a copy of $K_M$, and $\pi_1(X)$ permutes these components transitively.
\end{lemma}
\begin{proof}
The preimage of $K$ is a covering space of $K$. A loop in $K$ lifts to a loop in
$\widetilde X$ if and only if it is trivial in $\pi_1(X)$, so the component through a
base lift is the cover with fundamental group $M$. The deck group acts transitively on
the fibre over the base point, hence on the components.
\end{proof}

\section{Proof of \texorpdfstring{$(1)\Rightarrow(2)$}{(1) implies (2)}}\label{sec:necessity}

Throughout this section $K$ is a two-complex satisfying~(1), $G=\pi_1(K)$, and we use
the notation of Section~\ref{sec:fox}. The letter $k$ denotes $\Z$ or $\Fp$ for a prime~$p$.

We begin with a special case of a theorem of Strebel \cite[Cor.~1.8(ii)]{Strebel}; we
include a proof for completeness.
\begin{lemma}\label{lem:coinv}
Let $H$ be a torsion-free abelian normal subgroup of a group $Q$, and let
$f\colon U\to V$ be a homomorphism of free left $k[Q]$-modules of arbitrary ranks.
If $k[Q/H]\otimes_{k[Q]}f$ is injective, then $f$ is injective.
\end{lemma}
\begin{proof}
Restrict scalars to $R=k[H]$. Since $k[Q]$ is a free $R$-module on a set of coset
representatives, $U$ and $V$ are free $R$-modules. Let $\varepsilon\colon R\to k$ be the
augmentation, with kernel $I_H$. As $H$ is normal, $k[Q/H]\otimes_{k[Q]}U\cong U/I_HU
\cong k\otimes_RU$, and under these identifications $k[Q/H]\otimes f$ becomes
$k\otimes_Rf$. The ring $R$ is a commutative domain: every finite computation in $R$
takes place in the Laurent polynomial ring $k[t_1^{\pm1},\dots,t_s^{\pm1}]$ of a finitely
generated, hence free abelian, subgroup of~$H$.

\emph{Claim.} If $v_1,\dots,v_m\in V$ have $k$-linearly independent images
$\bar v_1,\dots,\bar v_m$ in $k\otimes_RV$, then they are $R$-linearly independent.

We induct on $m$, the case $m=0$ being empty. Since $\bar v_1\ne0$, some coordinate
functional $\lambda\colon V\to R$ with respect to an $R$-basis of $V$ satisfies
$\varepsilon(a)\ne0$, where $a:=\lambda(v_1)$; in particular $a\ne0$. Put
$w_j=av_j-\lambda(v_j)v_1$ for $2\le j\le m$. Their images
$\varepsilon(a)\bar v_j-\varepsilon(\lambda(v_j))\bar v_1$ are $k$-linearly independent,
because the $\bar v_j$ are and $\varepsilon(a)\ne0$ in the domain~$k$. By induction the
$w_j$ are $R$-linearly independent. If $\sum_jc_jv_j=0$ with $c_j\in R$, then, $R$
being commutative,
\[
0=a\sum_jc_jv_j-\lambda\Bigl(\sum_jc_jv_j\Bigr)v_1=\sum_{j\ge2}c_jw_j ,
\]
so $c_j=0$ for $j\ge2$. Then $c_1v_1=0$, and applying $\lambda$ gives $c_1a=0$, so $c_1=0$.

Now let $(e_\alpha)$ be an $R$-basis of $U$. The $\bar e_\alpha$ form a $k$-basis of
$k\otimes_RU$, so by injectivity of $k\otimes_Rf$ every finite subfamily of
$(\overline{f(e_\alpha)})$ is $k$-linearly independent. By the claim, every finite
subfamily of $(f(e_\alpha))$ is $R$-linearly independent, i.e.\ $f$ is injective.
\end{proof}

\subsection{Cycle tests}
\begin{definition}
A \emph{test} is a pair $(k,v)$ with $k\in\{\Z\}\cup\{\Fp:p\text{ prime}\}$ and
$v\in k[G]^{(J)}$. A normal subgroup $M\trianglelefteq G$ \emph{passes} the test
$(k,v)$ if $\partial_M(v_M)=0$ implies $v_M=0$; otherwise it \emph{fails} it. We call
$M$ \emph{good} if it passes every test.
\end{definition}
Since every element of $k[G/M]^{(J)}$ is of the form $v_M$, a normal subgroup $M$ is
good if and only if $\partial_M$ is injective over $\Z$ and over every $\Fp$, i.e.\
(Proposition~\ref{prop:fox}) $H_2(K_M;k)=0$ for all such~$k$.

\begin{lemma}\label{lem:pigeon}
Let $K=X_0\subset X_1\subset\cdots\subset X_m$ be a chain of two-complexes in which every
inclusion is zero on $\pi_2$, and let $M_r=\ker(G\to\pi_1X_r)$, so that
$1=M_0\subseteq M_1\subseteq\cdots\subseteq M_m$. For every test $(k,v)$ and
$0\le r<m$,
\[
\partial_{M_r}(v_{M_r})=0\ \Longrightarrow\ v_{M_{r+1}}=0 .
\]
Consequently each test is failed by at most one $M_r$, and any $m$ tests are all passed
by some $M_r$ with $0\le r\le m$.
\end{lemma}
\begin{proof}
Fix $r<m$ and write $X=X_r$, $X'=X_{r+1}$, $M=M_r$, $M'=M_{r+1}$. By
Lemma~\ref{lem:components}, the components of the preimages of $K$ in $\widetilde X$
and $\widetilde X'$ through the base points are copies of $K_M$ and $K_{M'}$, and the
based lift $\widetilde X\to\widetilde X'$ of the inclusion restricts to the covering map
$K_M\to K_{M'}$. On $2$-chains, the latter is the reduction
$k[G/M]^{(J)}\to k[G/M']^{(J)}$.

Assume $\partial_M(v_M)=0$. Then $v_M$ is a $2$-cycle of $K_M\subset\widetilde X$, and
its image $v_{M'}$ is a $2$-cycle of $\widetilde X'$. The map
$H_2(\widetilde X;k)\to H_2(\widetilde X';k)$ is zero: for $k=\Z$ it is the map
$\pi_2(X)\to\pi_2(X')$; for $k=\Fp$, since $H_1(\widetilde X)=H_1(\widetilde X')=0$,
the universal coefficient theorem identifies it naturally with
$(\pi_2(X)\to\pi_2(X'))\otimes\Fp=0$. Thus $v_{M'}$ is a null-homologous $2$-cycle in a
complex without $3$-cells, so $v_{M'}=0$.

If a test fails at $r$, then $v_{M_{r+1}}=0$, hence $v_{M_s}=0$ for all $s>r$ (as
$M_{r+1}\subseteq M_s$), so it passes at every $s>r$. Applying this to the smallest
index where the test fails shows that it fails at most once. Since $m$ tests have at
most $m$ failures among the $m+1$ indices, some index is failed by none.
\end{proof}

\begin{lemma}\label{lem:compact}
The set $\mathcal G$ of good normal subgroups of $G$ is nonempty and has a minimal
element with respect to inclusion.
\end{lemma}
\begin{proof}
Identify subsets of $G$ with points of the compact space $\{0,1\}^G$; the normal
subgroups form a closed subset $\mathcal N$. Fix a test $(k,v)$, and let $S\subset G$ be
a finite set containing the supports of all coordinates of $v$ and of $\partial_1v$.
Since $v_M$ and $\partial_M(v_M)=(\partial_1v)_M$ are obtained by summing coefficients
over left cosets $sM$, and $sM=s'M$ if and only if $s^{-1}s'\in M$, whether $M$ passes
$(k,v)$ depends only on $M\cap S^{-1}S$. Hence the set $\mathcal N_{(k,v)}$ of normal
subgroups passing $(k,v)$ is clopen in $\mathcal N$. By hypothesis~(1) and
Lemma~\ref{lem:pigeon}, applied to a chain of length $m$, any $m$ of the sets
$\mathcal N_{(k,v)}$ have a common point. By compactness,
$\mathcal G=\bigcap_{(k,v)}\mathcal N_{(k,v)}$ is nonempty, and it is closed.

If $\mathcal C\subseteq\mathcal G$ is totally ordered by inclusion, then
$\bigcap\mathcal C$ is the limit in $\{0,1\}^G$ of the net $\mathcal C$, ordered by
reverse inclusion: for each $g\in G$, membership of $g$ is eventually constant and equal
to membership in $\bigcap\mathcal C$. Hence $\bigcap\mathcal C\in\mathcal G$, and Zorn's
lemma gives a minimal element.
\end{proof}

\begin{proof}[Proof of $(1)\Rightarrow(2)$]
Let $N$ be a minimal element of $\mathcal G$. Then $H_2(K_N;k)=0$ for $k=\Z$ and every
$k=\Fp$. By the universal coefficient theorem,
$\mathrm{Tor}(H_1(K_N),\Fp)\subseteq H_2(K_N;\Fp)=0$ for every $p$, so
$H_1(K_N)=N/[N,N]$ is torsion-free.

Apply Lemma~\ref{lem:coinv} with $Q=G/[N,N]$, $H=N/[N,N]$ and $f=\partial_{[N,N]}$ over
$k$. Since $C_*(K_N;k)=k[G/N]\otimes_{k[Q]}C_*(K_{[N,N]};k)$, we have
$k[Q/H]\otimes_{k[Q]}f=\partial_N$, which is injective. Hence $\partial_{[N,N]}$ is
injective over every $k$, i.e.\ $[N,N]\in\mathcal G$. Minimality gives $[N,N]=N$. Thus
$H_1(K_N)=0$ and $H_2(K_N)=0$, and $K_N$ is a connected acyclic regular cover of~$K$.
\end{proof}

\section{Reduction to the acyclic cover}\label{sec:descent}

Let $p\colon D\to K$ be a connected acyclic regular cover, $G=\pi_1(K)$,
$N=p_*\pi_1(D)$ and $H=G/N$. Section~\ref{sec:acyclic} builds chains over $D$ in which
$\pi_1(D)$ dies immediately. This section transfers them to~$K$.

\begin{proposition}[Descent]\label{prop:descent}
Let $D=L_0\subsetneq L_1\subsetneq\cdots\subsetneq L_n$ be two-complexes with all
inclusions zero on $\pi_2$, such that $\pi_1(D)\to\pi_1(L_i)$ is trivial for $i\ge1$.
Let $E_i=K\cup_pL_i$ be the pushout of $K\xleftarrow{p}D\subset L_i$, i.e.\ $K$
together with the cells of $L_i\smallsetminus D$, attached by composing their attaching
maps with $p$ on~$D$. Then $K=E_0\subsetneq E_1\subsetneq\cdots\subsetneq E_n$ are
two-complexes and all inclusions are zero on $\pi_2$. If $K$ and all $L_i$ are finite,
so are all~$E_i$.
\end{proposition}
\begin{proof}
\emph{Cells and $\pi_1$.} The cells of $E_i$ are those of $K$ and those of
$L_i\smallsetminus D$, so the $E_i$ are strictly nested two-complexes. Put
$\Pi_i=\pi_1(L_i)$. By van Kampen's theorem $\pi_1(E_i)$ is the pushout of
$G\leftarrow\pi_1(D)\rightarrow\Pi_i$. The first map has normal image $N$ and the second
is trivial, so
\[
\Lambda_i:=\pi_1(E_i)=(G/N)*\Pi_i=H*\Pi_i .
\]

\emph{$\pi_2$ is generated from $L_i$.} Write $\Lambda=\Lambda_i$ and $\Pi=\Pi_i$. By
Lemma~\ref{lem:components}, in $\widetilde E_i$: the preimage of $K$ is a disjoint union
of copies of $K_N=D$ indexed by $\Lambda/H$ (the kernel of $G\to\Lambda$ is $N$, since
$H$ embeds in $H*\Pi$); the preimage of $L_i$ is a disjoint union of copies of
$\widetilde L_i$ indexed by $\Lambda/\Pi$; and the preimage of $D$ is a disjoint union
of copies of $D$ indexed by $\Lambda$, as $\pi_1(D)$ dies in $\Lambda$. These are
subcomplexes, the first two cover $\widetilde E_i$ and intersect in the third, so we
have a short exact sequence of cellular chain complexes
\[
0\to\bigoplus_{\Lambda}C_*(D)\to\Bigl(\bigoplus_{\Lambda/H}C_*(D)\Bigr)\oplus
\Bigl(\bigoplus_{\Lambda/\Pi}C_*(\widetilde L_i)\Bigr)\to C_*(\widetilde E_i)\to0 .
\]
Since $H_1(D)=H_2(D)=0$, the long exact homology sequence gives an isomorphism
$\bigoplus_{\Lambda/\Pi}H_2(\widetilde L_i)\to H_2(\widetilde E_i)$, that is,
\begin{equation}\label{eq:descent-pi2}
\Z[\Lambda]\otimes_{\Z[\Pi]}\pi_2(L_i)\ \xrightarrow{\ \cong\ }\ \pi_2(E_i),
\end{equation}
induced by $L_i\to E_i$. In particular
$\pi_2(E_i)=\Z[\Lambda]\cdot\operatorname{im}\pi_2(L_i)$.

\emph{Zero maps.} The covering $p$ is bijective on $\pi_2$, and the composite
$D\xrightarrow{p}K\subset E_1$ equals $D\subset L_1\to E_1$, which is zero on $\pi_2$;
hence $K\subset E_1$ is zero on $\pi_2$. For $i\ge1$, the square formed by
$L_i\subset L_{i+1}$ and $E_i\subset E_{i+1}$ commutes, its top arrow is zero on
$\pi_2$, and $\pi_2(E_i)$ is generated by the image of $\pi_2(L_i)$; so
$E_i\subset E_{i+1}$ is zero on $\pi_2$ by Lemma~\ref{lem:toolbox}(c).
\end{proof}

\begin{proof}[Proof of $(2)\Rightarrow(1)$, given Proposition~\ref{prop:relative}]
Apply Proposition~\ref{prop:relative} to the acyclic cover $D$ and then
Proposition~\ref{prop:descent}. If $K$ is finite and acyclic, take $D=K$ and $p$ the
identity; then $E_i=L_i$, which are finite by Proposition~\ref{prop:relative}.
\end{proof}

\section{The chain over an acyclic complex}\label{sec:acyclic}

\begin{proposition}\label{prop:relative}
Let $D$ be an acyclic two-complex and $n\ge1$. There are two-complexes
\[
D=L_0\subsetneq L_1\subsetneq\cdots\subsetneq L_n
\]
with all inclusions zero on $\pi_2$ and $\pi_1(D)\to\pi_1(L_i)$ trivial for $i\ge1$.
If $D$ is finite, all $L_i$ can be chosen finite.
\end{proposition}

\subsection{Presentations extending \texorpdfstring{$D$}{D}}
Fix a maximal tree $\mathcal T_D$ of the $1$-skeleton of $D$ and a base vertex.
Collapsing $\mathcal T_D$ gives a presentation $\pres{x_i\ (i\in I)\mid r_j\ (j\in J)}$,
where the $x_i$ are the edges of $D$ outside $\mathcal T_D$; we fix a total order on
$I$. Every word in the $x_i$ determines a based edge loop in $D$, going through
$\mathcal T_D$. A \emph{presentation extending $D$} is a complex $P\supseteq D$ obtained
by attaching loops at the base vertex (\emph{extra generators}) and $2$-cells along words
in the $x_i$ and the extra generators (\emph{extra relators}). The $x_i$ and $r_j$ are
the \emph{core} generators and relators. Relators are \emph{labelled}: equal words with
different labels are different $2$-cells. We write $G(P)=\pi_1(P)$, and $\cx{\mathcal P}$
for the presentation complex of a presentation $\mathcal P$. Since
$P\simeq P/\mathcal T_D$, which is the presentation complex of the corresponding
presentation, we freely pass between the two.

Since $D$ is acyclic and $D/\mathcal T_D$ has a single vertex, its cellular boundary
(Remark~\ref{rem:fox})
\begin{equation}\label{eq:exponent}
\Theta\colon\Z^{(J)}\to\Z^{(I)},\qquad \Theta_{ij}=\text{exponent sum of }x_i\text{ in }r_j,
\end{equation}
is an isomorphism, since its kernel is $H_2(D)$ and its cokernel is $H_1(D)$. Note that
$\Theta^{-1}$ maps finitely supported vectors to finitely supported vectors.

If $I=\varnothing$, then $D/\mathcal T_D$ is a wedge of $|J|$ spheres, so $J=\varnothing$
since $H_2(D)=0$, and $D$ is contractible. We then take for $L_i$ the complex $D$ with
$i$ pairs $(e^1,e^2)$ attached, $e^2$ along the loop $e^1$; all $\pi_2$ vanish.
Henceforth $I\ne\varnothing$.

\subsection*{Outline}
We write $[u,v]=uvu^{-1}v^{-1}$. The chain is built by induction on~$n$. We start from a
pair $D\subset Y_D$ that kills $\pi_1(D)$ (Lemma~\ref{lem:initial}). Given a chain
$D=L_0\subset\cdots\subset L_m$, we apply to each term an operation $T$ that turns every
extra generator $z$ into a commutator $[a_z,b_z]$ of two new generators, corrects every
extra relator by core relators so that it becomes a product of commutators, and then
replaces it by a \emph{block}. By the defining property \ref{B2} of blocks, $\pi_2$ of
$T(P)$ is generated by the image of $\pi_2(P)$, so $T$ preserves inclusions that are zero
on $\pi_2$ (Lemma~\ref{lem:structural}). We then append the \emph{terminal extension}
$\Tp(L_m)$, obtained by capping all generators $a_z,b_z$. It kills $\pi_1(L_m)$, and its
second homology comes from ``capped blocks'', which are Cockcroft by \ref{B3}
(Lemma~\ref{lem:terminal}). Each such pass lengthens the chain by one.

\subsection{The initial pair}
Define
\begin{equation}\label{eq:initial}
\begin{split}
Y_D=\pres{x_i,a_i,b_i\ (i\in I)\mid\ &r_j\ (j\in J),\quad x_i[a_i,b_i]^{-1}\ (i\in I),\\
&[a_i,a_k],\,[a_i,b_k],\,[b_i,a_k],\,[b_i,b_k]\ (i<k)},
\end{split}
\end{equation}
a presentation extending $D$.
\begin{lemma}\label{lem:initial}
$G(Y_D)$ is free abelian on the $a_i,b_i$, and $x_i=1$ in it. The complex $Y_D$
is Cockcroft, and $D\subset Y_D$ is trivial on $\pi_1$ and zero on $\pi_2$.
\end{lemma}
\begin{proof}
\emph{The group.} Eliminate $x_i=U_i:=[a_i,b_i]$. The cross-commutators make the
subgroups $\pres{a_i,b_i}$ commute pairwise, so the $U_i$ commute, and $r_j$ becomes
$r_j(U)=\prod_iU_i^{\Theta_{ij}}$. For $c=\Theta^{-1}e_i\in\Z^{(J)}$ we get
\[
U_i=\prod_jr_j(U)^{c_j}=1 .
\]
So $x_i=1$ and all $a_i,b_i$ commute. Conversely, sending $a_i,b_i$ to a basis of the
free abelian group on $\{a_i,b_i\}$ and $x_i$ to $0$ satisfies every relator.

\emph{Cockcroft.} Eliminating the $x_i$ by the relators $x_iU_i^{-1}$ is a Tietze
transformation, hence a homotopy equivalence \cite[Ch.~I]{HMS}, from $Y_D/\mathcal T_D$
to the presentation complex $Y'$ with generators $a_i,b_i$ and relators $r_j(U)$ and the
cross-commutators. All these relators have zero exponent sums, so the cellular boundary
$\partial_2$ of $Y'$ vanishes and $H_2(Y')$ is free on the $2$-cells. Let $\mathbb T$ be
the union of the finite subtori of $\prod_{i\in I}(S^1_{a_i}\times S^1_{b_i})$, with its
product cell structure. It is a $K(\Z^{(2I)},1)$, and $H_2(\mathbb T)$ is free on the
coordinate squares. Map $Y'\to\mathbb T$ by sending generators to the corresponding
circles, each cross-commutator disk onto its square with degree one, and the disk of
$r_j(U)$, whose boundary is a concatenation of loops $U_i^{\pm1}$, onto the
concatenation of the corresponding $(a_i,b_i)$-squares. On $H_2$:
\[
[r_j(U)]\longmapsto\sum_i\Theta_{ij}\,[(a_i,b_i)\text{-square}],
\qquad[\text{cross-commutator}]\longmapsto[\text{its square}].
\]
A class in the kernel has zero cross-commutator coefficients, and its vector
$\lambda\in\Z^{(J)}$ of remaining coefficients satisfies $\Theta\lambda=0$, so
$\lambda=0$. Thus $H_2(Y')\to H_2(\mathbb T)$ is injective, and $Y'\simeq Y_D$ is
Cockcroft by Lemma~\ref{lem:toolbox}(d).

\emph{The inclusion.} $D\subset Y_D$ is trivial on $\pi_1$, as $\pi_1(D)$ is generated by
the $x_i$, and $D$ is Cockcroft since $H_2(D)=0$; apply Lemma~\ref{lem:toolbox}(b).
\end{proof}

\subsection{Blocks}\label{sec:blocktheorem}
For $q\ge1$ let $\Sigma_q$ be the closed oriented surface of genus $q$ with its standard
one-vertex cell structure: edges $s_1,t_1,\dots,s_q,t_q$ and one $2$-cell attached along
$\omega_q=\prod_{h=1}^q[s_h,t_h]$.

\begin{theorem}[Blocks]\label{thm:blocks}
For every $q\ge1$ there is a finite presentation
\[B_q=\pres{s_1,t_1,\dots,s_q,t_q,Z_q\mid\beta_{q,1},\dots,\beta_{q,m_q}},\]
with $m_q\ge1$ and $Z_q$ a finite list of \emph{internal} generators, such that:
\begin{enumerate}[label={\normalfont\bfseries(B\arabic*)},leftmargin=14mm]
\item\label{B1} \emph{Marking.} There is a cellular map $b_q\colon\Sigma_q\to\cx{B_q}$
sending each edge to the generator of the same name.
\item\label{B2} \emph{Replacement creates no new $\pi_2$.} Let $E=\pres{X\mid R,r}$ be a
presentation, and suppose $r=\prod_{h=1}^q[u_h,v_h]$ in the free group on $X$. Replace
the relator $r$ by a block with fresh internal generators $Z_r$:
\begin{equation}\label{eq:replace}
E'=\pres{X,Z_r\mid R,\ \beta_{q,j}(u_1,v_1,\dots,u_q,v_q,Z_r)\ (1\le j\le m_q)} .
\end{equation}
Let $f\colon\cx{E}\to\cx{E'}$ be the \emph{structural map}: the identity on the
$1$-skeleton of $\cx{E}$ and on the $2$-cells of $R$, and on the $2$-cell $r$ the
filling of $\prod_h[u_h,v_h]$ given by $b_q$. Then
\begin{equation}\label{eq:block-generation}
\pi_2(E')=\Z[G(E')]\cdot f_*\pi_2(E).
\end{equation}
The same holds for the simultaneous replacement of any set of relators, each with its
own fresh internal generators.
\item\label{B3} \emph{Capped blocks are Cockcroft.} The presentation
$\pres{Z_q\mid\beta_{q,j}(1,\dots,1,Z_q)\ (1\le j\le m_q)}$ is Cockcroft.
\end{enumerate}
\end{theorem}
Section~\ref{sec:blocks} proves this theorem. Property \ref{B2} is what makes $T$
preserve inclusions that are zero on $\pi_2$, and \ref{B3} is what makes the terminal
extension Cockcroft.

\subsection{The operations \texorpdfstring{$T$ and $\Tp$}{T and T+}}
Let $P$ be a presentation extending $D$. Fix once and for all, for every word $w$ with
zero exponent sum in every generator, an expression $w=\prod_{h=1}^q[u_h,v_h]$ in the
free group with $q\ge1$. Such expressions exist since such $w$ lie in the commutator
subgroup; for instance, repeatedly use
$p\,\ell\,m\,t=[p\ell p^{-1},pmp^{-1}]\cdot p\,m\,\ell\,t$ to sort the letters of $w$
according to a fixed order on the generators, and use $[1,1]$ if no commutator occurs.
Define $T(P)$ by the following three rules.
\begin{enumerate}[label=\textbf{\arabic*.}]
\item \emph{Generators become commutators.} For each extra generator $z$ of $P$ take a
fresh pair of generators $a_z,b_z$ and substitute $z\mapsto[a_z,b_z]$ in every extra
relator. Core generators and core relators are untouched. Call $r^*$ the substituted
form of an extra relator~$r$.
\item \emph{Correction by core relators.} Let $e\in\Z^{(I)}$ be the vector of exponent
sums of $r^*$ in the $x_i$, and $c=-\Theta^{-1}e\in\Z^{(J)}$. Replace $r^*$ by
\begin{equation}\label{eq:correction}
r^0=r^*\prod_{j\in\supp(c)}r_j^{\,c_j}\qquad(\text{in increasing order of }j),
\end{equation}
which has zero exponent sum in every generator: in the $x_i$ by the choice of $c$, and in
the $a_z,b_z$ because they only occur inside the commutators $[a_z,b_z]$.
\item \emph{Blocks.} Write $r^0=\prod_{h=1}^q[u_h,v_h]$ using the fixed expression, and
replace the relator $r^0$ by the block~\eqref{eq:replace}, with fresh internal generators
$Z_r$ for each label~$r$.
\end{enumerate}
The core is unchanged, so $T(D)=D$. The words $u_h,v_h$ involve only the $x_i$ and the
$a_z,b_z$. Composing the maps of the three rules (the third being the structural map of
\ref{B2}) gives a \emph{structural map} $\eta_P\colon P\to T(P)$, with $x_i\mapsto x_i$
and $z\mapsto[a_z,b_z]$. The \emph{terminal extension} is
\begin{equation}\label{eq:terminal}
\Tp(P)=T(P)\ \cup\ \{\text{$2$-cells along }a_z\text{ and along }b_z:\ z\text{ an extra generator of }P\};
\end{equation}
internal generators of blocks are not capped.

\emph{Naturality.} All choices in $T$ (the names $a_z,b_z$, the corrections, the
factorizations, the internal generators) depend only on the labelled generator or relator
concerned. Hence a labelled inclusion $j\colon P\subset P'$ induces a labelled inclusion
$T(j)\colon T(P)\subset T(P')$ with $T(j)\circ\eta_P=\eta_{P'}\circ j$. A new extra
generator yields a new pair $a_z,b_z$, and a new extra relator yields a new block, which
has $m_q\ge1$ relators; so $T$ preserves strict inclusions.

\begin{lemma}\label{lem:structural}
For every presentation $P$ extending $D$,
\begin{equation}\label{eq:generation}
\pi_2(T(P))=\Z[G(T(P))]\cdot\eta_{P*}\pi_2(P).
\end{equation}
Consequently, if $P$ is Cockcroft then so is $T(P)$, and if $j\colon P\subset P'$ is zero
on $\pi_2$ then so is $T(j)$.
\end{lemma}
\begin{proof}
If $f\colon A\to B$ and $g\colon B\to C$ satisfy the generation property
$\pi_2(B)=\Z[\pi_1B]\cdot f_*\pi_2(A)$ and similarly for $g$, then so does $gf$,
because $g_*(\gamma s)=g_*(\gamma)\,g_*(s)$. It therefore suffices to check generation for
each rule.

\emph{Rule 1.} Let $P_1'$ be obtained from $P$ by adding, for each extra generator $z$,
a loop $b_z$, a loop $a_z$ and a $2$-cell along $a_zb_za_z^{-1}(zb_z)^{-1}$. This relator
says $z=[a_z,b_z]$, and eliminating each $z$ by it is a Tietze transformation
\cite[Ch.~I]{HMS}, giving the presentation of rule~1; under this homotopy equivalence the
inclusion $P\subset P_1'$ becomes the map of rule~1. Now $P_1'$ is a graph of spaces
\cite{ScottWall}: the vertex space is $P\vee\bigvee_zS^1_{b_z}$, and each new $2$-cell
together with $a_z$ is a cylinder $S^1\times[0,1]$ glued along $b_z$ and $zb_z$. Both
loops have infinite order in $G(P)*F(\{b_z\})$, as is seen by killing $G(P)$. So
$G(P_1')$ is an HNN extension along infinite cyclic subgroups, in which the vertex group
embeds \cite[Ch.~IV, \S2]{LyndonSchupp}. Consequently $\widetilde{P_1'}$ is a tree of
spaces whose vertex spaces $V_\alpha$ are copies of the universal cover of
$P\vee\bigvee_zS^1$ and whose edge spaces are strips $\mathbb R\times[0,1]$. Since
$H_*(\mathbb R\times[0,1],\mathbb R\times\{0,1\})$ is concentrated in degree~$1$,
excision gives $H_2(\widetilde{P_1'},\bigsqcup_\alpha V_\alpha)=0$, so
$\bigoplus_\alpha H_2(V_\alpha)\to H_2(\widetilde{P_1'})$ is onto. By the same argument
applied to the wedge with circles (a tree of copies of $\widetilde P$ joined by lines),
each $H_2(V_\alpha)$ is generated by translates of $\pi_2(P)$.

\emph{Rule 2.} The attaching maps of $r^*$ and $r^0$ are homotopic in the complex
obtained by removing all extra relator cells, since the correcting word is null-homotopic
in $D$. Hence the map of rule~2 is a homotopy equivalence \cite[Prop.~0.18]{Hatcher}.

\emph{Rule 3.} This is exactly \ref{B2}.

The consequences follow from~\eqref{eq:generation} and Lemma~\ref{lem:toolbox}(c): if
$P$ is Cockcroft, so is $T(P)$; and if $j_*=0$ on $\pi_2$, then
$T(j)_*\eta_{P*}=\eta_{P'*}j_*=0$, so $T(j)_*=0$.
\end{proof}

\begin{lemma}\label{lem:terminal}
Suppose that $P$ is Cockcroft and that all core generators $x_i$ are trivial in $G(P)$.
Then the $x_i$ are trivial in $G(T(P))$ and $G(\Tp(P))$, the inclusion
$T(P)\subset\Tp(P)$ is zero on $\pi_2$, and $\Tp(P)$ is Cockcroft.
\end{lemma}
\begin{proof}
The $x_i$ remain trivial because $\eta_P$ fixes them.

\emph{The inclusion.} The composite $P\to T(P)\subset\Tp(P)$ is trivial on $\pi_1$: the
$x_i$ are trivial by hypothesis, and each extra generator $z$ maps to $[a_z,b_z]$, a
commutator of two capped loops. Since $P$ is Cockcroft, the composite is zero on $\pi_2$
by Lemma~\ref{lem:toolbox}(b). By~\eqref{eq:generation} and Lemma~\ref{lem:toolbox}(c),
the inclusion $T(P)\subset\Tp(P)$ is zero on $\pi_2$.

\emph{Cockcroft.} Since $D$ is acyclic, the quotient map
$\kappa\colon\Tp(P)\to\Tp(P)/D$ is an isomorphism on $H_2$ (long exact sequence of the
pair). In $\Tp(P)/D$, all $x_i$ and core relators have disappeared, and the remaining
cells are the loops $a_z,b_z$ with their capping disks, and the blocks. Eliminating the
$a_z,b_z$ by their capping disks is a homotopy equivalence \cite[Ch.~I]{HMS}. As the words
$u_h,v_h$ only involve the $x_i$ and the $a_z,b_z$, every block relator
$\beta_{q,j}(u,v,Z_r)$ becomes $\beta_{q,j}(1,\dots,1,Z_r)$, and distinct blocks have
disjoint internal generators. Therefore
\[
\Tp(P)/D\ \simeq\ \bigvee_r\cx{\pres{Z_r\mid\beta_{q(r),j}(1,\dots,1,Z_r)\ (1\le j\le m_{q(r)})}} ,
\]
the wedge being over the extra relators $r$ of $P$. Each wedge summand is Cockcroft by
\ref{B3}, hence so is the wedge (Lemma~\ref{lem:toolbox}(d)), and $\Tp(P)$ is Cockcroft
by Lemma~\ref{lem:toolbox}(d) applied to~$\kappa$.
\end{proof}

\subsection{Proof of Proposition~\ref{prop:relative}}
We construct chains $D=L_0\subsetneq L_1\subsetneq\cdots\subsetneq L_m$ of presentations
extending $D$, related by labelled inclusions, such that
\begin{equation}\label{eq:inductive}
\parbox{0.85\textwidth}{\centering\itshape every $L_i$ is Cockcroft, every inclusion is
zero on $\pi_2$, and all $x_i$ are trivial in $G(L_i)$ for $i\ge1$.}
\end{equation}
By Lemma~\ref{lem:initial}, $(D,Y_D)$ is such a chain with $m=1$. Given such a chain, set
\begin{equation}\label{eq:induction}
(D,L_1,\dots,L_m)\ \longmapsto\ (D,\ T(L_1),\dots,T(L_m),\ \Tp(L_m)).
\end{equation}
By naturality, this is a strictly increasing chain of labelled inclusions; note that
$\Tp(L_m)\ne T(L_m)$ since $L_m\supseteq Y_D$ has extra generators. By
Lemma~\ref{lem:structural}, the inclusions $D=T(D)\subset T(L_1)$ and
$T(L_i)\subset T(L_{i+1})$ are zero on $\pi_2$, and every $T(L_i)$ is Cockcroft. By
Lemma~\ref{lem:terminal}, applicable because $L_m$ satisfies~\eqref{eq:inductive} and
$m\ge1$, the $x_i$ are trivial in $G(T(L_i))$ and $G(\Tp(L_m))$, the complex $\Tp(L_m)$ is
Cockcroft, and $T(L_m)\subset\Tp(L_m)$ is zero on $\pi_2$. So the new chain
satisfies~\eqref{eq:inductive}. Applying~\eqref{eq:induction} $n-1$ times to $(D,Y_D)$
gives a chain of length $n$; as $\pi_1(D)$ is generated by the $x_i$, it is killed in
every $L_i$ with $i\ge1$.

If $D$ is finite, then $I$ and $J$ are finite, and each pass introduces finitely many
generators and finitely many finite blocks; so all $L_i$ are finite. \qed

\section{Construction of the blocks}\label{sec:blocks}

This section proves Theorem~\ref{thm:blocks}. We combine two constructions. Davis's real
moment-angle complexes realize a prescribed finite flag complex as every vertex link of a
cube complex \cite[Example~2.2]{Davis}. The relative hyperbolization of
Davis--Januszkiewicz--Weinberger \cite[\S2]{DJW} cones off boundary components,
hyperbolizes, and truncates the cone points. We adapt two of their arguments, the radial
retraction \cite[Lemma~2.2]{DJW} and the CAT(0) property of a branched cover
\cite[Lemma~2.6]{DJW}, to show that $\pi_2$ does not grow when an \emph{arbitrary}
two-complex is glued into the truncated corner (Lemma~\ref{lem:geometric-pushout}).

\subsection{The cube complex and its truncation}
A simplicial complex is \emph{flag} if every finite set of pairwise adjacent vertices spans
a simplex. For a finite flag complex $L$ on the vertex set $\{1,\dots,d\}$ put
\begin{equation}\label{eq:cubes}
C(L)=\bigcup_{\sigma\in L}[-1,1]^{\sigma}\times\{-1,1\}^{\{1,\dots,d\}\smallsetminus\sigma}
\ \subset\ [-1,1]^d,
\end{equation}
a cube complex whose cubes are indexed by a simplex $\sigma$ and a choice of signs on
the remaining coordinates. Every vertex link of $C(L)$ is isomorphic to $L$. We use
Gromov's link condition: a cube complex whose vertex links are flag is locally CAT(0),
so its universal cover is CAT(0) and in particular contractible
\cite[Thms.~II.5.20 and II.4.1(2)]{BH}. For cube complexes that are not locally finite,
which occur in the proof of Lemma~\ref{lem:geometric-pushout}, see
\cite[Thm.~C.9 and its proof; Thms.~C.4, A.7]{Leary}.

Choose a finite triangulation of $\Sigma_q$ in which the $2q$ standard loops are edge
paths, and let $L_q$ be its barycentric subdivision. It is flag, since its simplices are
the chains of faces of the original triangulation, and pairwise comparable faces form a
chain. Number the vertices of $L_q$, set $C_q=C(L_q)$ and $v=(1,\dots,1)$, and cut off
the corner at $v$:
\begin{equation}\label{eq:cut}
M_q=C_q\cap\Bigl\{\textstyle\sum_h(1-x_h)\ge\tfrac12\Bigr\}.
\end{equation}
Since every other vertex of $C_q$ is at $\ell^1$-distance at least $2$ from $v$, only
the corner at $v$ is removed. The \emph{cut surface}
$C_q\cap\{\sum_h(1-x_h)=\tfrac12\}$ is a copy of the link of $v$, hence of $\Sigma_q$
(triangulated as $L_q$), and the removed part is the open cone on it. We regard
$\Sigma_q$ as the cut surface, the standard loops being edge paths in it.

\begin{lemma}\label{lem:spine}\
\begin{enumerate}[label=\textup{(\roman*)},itemsep=2pt]
\item $C_q=M_q\cup_{\Sigma_q}\Cone(\Sigma_q)$ is a finite $3$-dimensional cube complex all
of whose vertex links are isomorphic to $L_q$. It is aspherical.
\item $M_q$ collapses onto a finite two-dimensional subcomplex.
\item The inclusion $\Sigma_q\subset M_q$ maps the fundamental class $[\Sigma_q]$ to $0$ in
$H_2(M_q)$.
\end{enumerate}
\end{lemma}
\begin{proof}
(i) This follows from the discussion above, as $\dim L_q=2$.

(iii) Since $L_q$ is a closed surface, every edge $\tau$ of $L_q$ lies in exactly two
triangles; hence every square of $C_q$ lies in exactly two $3$-cubes. Orient $\Sigma_q$,
and let $o(\sigma)=\pm1$ compare the orientation of a triangle $\sigma$ with the one given
by increasing order of its vertices. Orient each cube $[-1,1]^\sigma\times\{\epsilon\}$
by increasing order of coordinates, and give it the coefficient
$o(\sigma)\prod_{h\notin\sigma}\epsilon_h$. Let a square $[-1,1]^\tau\times\{\eta\}$ lie
in the cubes indexed by $\sigma_i=\tau\cup\{k_i\}$, $i=1,2$, and let $p_i$ be the number
of elements of $\tau$ smaller than $k_i$. The square appears in the boundary of the
$i$-th cube with coefficient $\eta_{k_1}\eta_{k_2}E\cdot o(\sigma_i)(-1)^{p_i}$, where
$E=\prod_{h\notin\sigma_1\cup\sigma_2}\eta_h$; and $o(\sigma_1)(-1)^{p_1}+o(\sigma_2)(-1)^{p_2}=0$
because $\tau$ cancels in the fundamental cycle of $\Sigma_q$. So these coefficients
define a $3$-cycle $\zeta$ of $C_q$. Now consider the exact sequence
\[
H_3(C_q)\to H_3(C_q,M_q)\xrightarrow{\ \delta\ }H_2(M_q).
\]
By excision $H_3(C_q,M_q)\cong H_3(\Cone\Sigma_q,\Sigma_q)\cong H_2(\Sigma_q)\cong\Z$,
and under this identification $\delta$ becomes the map induced by $\Sigma_q\subset M_q$.
The cubes containing $v$ are those with all signs $+1$, and their coefficients in
$\zeta$ are the $o(\sigma)$; so $\zeta$ maps to a generator of $H_3(C_q,M_q)$. Hence
$\delta=0$, and $[\Sigma_q]\mapsto0$.

(ii) $M_q$ is a polyhedral complex whose $3$-cells are the cubes of $C_q$, those
containing $v$ being truncated. Consider the graph whose vertices are these $3$-cells,
two being adjacent when they share a $2$-cell. It is connected: the cubes containing a
given vertex $w\in\{\pm1\}^d$ correspond to the triangles of $L_q$, and are connected
because $L_q$ is a connected closed surface; and two vertices $w,w'$ differing in one
coordinate $k$ lie in a common cube, namely any cube indexed by a triangle containing
$k$. Choose a spanning tree, rooted at a truncated cube. A triangle of the cut surface
lies in only one $3$-cell, so we can collapse the root across it. Then collapse each other
$3$-cell, in order of distance to the root, across the $2$-cell it shares with its parent:
that $2$-cell lies in exactly two $3$-cells and the parent has been removed, so it is a
free face. All $3$-cells disappear.
\end{proof}

\subsection{The presentation \texorpdfstring{$B_q$}{Bq} and the marking (B1)}
Let $S$ be the two-dimensional subcomplex of Lemma~\ref{lem:spine}(ii), and let
$\rho\colon M_q\to S$ be the deformation retraction given by the collapse. By cellular
approximation \cite[Thm.~4.8]{Hatcher}, $\rho|_{\Sigma_q}$ is homotopic to a cellular map
$c\colon\Sigma_q\to S$, where $\Sigma_q$ carries its one-vertex structure. Contract a
maximal tree of $S$, call $z_1,\dots,z_k$ the remaining oriented edges, let $w_F(z)$ be
the boundary word of each $2$-cell $F$ of $S$, and let $\mu_h(z),\nu_h(z)$ be the words
read along $c(s_h),c(t_h)$. Put
\begin{equation}\label{eq:block-presentation}
\begin{split}
B_q=\pres{s_1,t_1,\dots,s_q,t_q,z_1,\dots,z_k\mid\ &w_F(z)\ (F\text{ a $2$-cell of }S),\\[-2pt]
&s_h\mu_h(z)^{-1},\ t_h\nu_h(z)^{-1}\ (1\le h\le q)} ,
\end{split}
\end{equation}
with $Z_q=(z_1,\dots,z_k)$. The relators $s_h\mu_h^{-1},t_h\nu_h^{-1}$ form
generator--relator pairs, so the inclusion $S/\text{tree}\subset\cx{B_q}$ is a homotopy
equivalence, and $\cx{B_q}\simeq M_q$. Define $b_q$ on the $1$-skeleton by $s_h\mapsto
s_h$, $t_h\mapsto t_h$. The loop $\omega_q(s,t)$ is homotopic in $\cx{B_q}$, across the
cells $s_h\mu_h^{-1},t_h\nu_h^{-1}$, to $\omega_q(\mu,\nu)$. This is the image under
$c$ of the boundary of the $2$-cell of $\Sigma_q$, so it bounds the image under $c$ of
that $2$-cell. Following the homotopy and then this disk defines $b_q$ on the $2$-cell,
and gives \ref{B1}. By construction, under $\cx{B_q}\simeq M_q$ the map $b_q$ corresponds to the
inclusion $\Sigma_q\subset M_q$ up to homotopy.

\subsection{Generation in the pushout: proof of (B2)}
Let $E$, $r=\prod_h[u_h,v_h]$ and $E'$ be as in \ref{B2}, and let
$a\colon\Sigma_q\to\cx{E}$ be the cellular map sending $s_h,t_h$ to the loops
$u_h,v_h$ and the $2$-cell of $\Sigma_q$ onto the $2$-cell $r$ with degree one (through a
collar realizing the free reduction of $\prod_h[u_h,v_h]$ to $r$). Consider the double
mapping cylinder
\begin{equation}\label{eq:block-pushout}
W=\cx{E}\cup_a(\Sigma_q\times[0,1])\cup_{b_q}\cx{B_q}.
\end{equation}
In $W$, the $3$-cell (polygon)$\times(0,1)$ has the $2$-cell $r$ as a free face; removing
both is an elementary collapse. The remaining cylinder cells are an edge
(vertex)$\times(0,1)$ joining the two base points, and $2$-cells
(edge)$\times(0,1)$ identifying $s_h,t_h$ with $u_h,v_h$. Contracting the edge and
eliminating $s_h,t_h$ by these $2$-cells gives exactly the presentation complex
of~\eqref{eq:replace}. Under this homotopy equivalence
$W\simeq\cx{E'}$, the inclusion $\cx{E}\subset W$ corresponds to the structural map $f$.
Moreover, since double mapping cylinders are homotopy invariant
\cite[Ch.~0]{Hatcher}, and $b_q$ corresponds to the inclusion of the cut surface, $W$ is
homotopy equivalent rel $\cx{E}$ to the space in the following lemma, with $X=\cx{E}$.
Three-dimensional spaces occur only in this proof: every complex in our chains is a
two-complex.

\begin{lemma}\label{lem:geometric-pushout}
Let $X$ be a two-complex, $a\colon\Sigma_q\to X$ a cellular map, and
$W=X\cup_a(\Sigma_q\times[0,1])\cup_{\Sigma_q}M_q$, where $\Sigma_q\times\{0\}$ is glued
to $X$ by $a$ and $\Sigma_q\times\{1\}$ to the cut surface of $M_q$. Then
$\pi_1(X)\to\pi_1(W)$ is injective, and $\pi_2(W)$ is generated over $\Z[\pi_1W]$ by the
image of $\pi_2(X)$.
\end{lemma}
\begin{proof}
Collapsing $X$ to a point turns $\Sigma_q\times[0,1]$ into $\Cone(\Sigma_q)$, so
$W/X=C_q$; let $\kappa\colon W\to C_q$ be the quotient map. It is a homeomorphism from
$W\smallsetminus X$ onto $C_q\smallsetminus\{v\}$.

\emph{Step 1: $\pi_1$-injectivity (cf.\ \cite[Lemma~2.2]{DJW}).} Let
$\widehat W\to W$ be the pullback of the universal cover $\widetilde C_q\to C_q$ along
$\kappa$. It is obtained from $\widetilde C_q$ by replacing each lift $v_\lambda$ of $v$,
together with its open cone neighbourhood, by a copy
$X_\lambda\cup_a(\Sigma_q\times[0,1])$; in particular it is connected. Fix a lift $v_0$.
Since $\widetilde C_q$ is CAT(0), each point $y$ outside the open cone neighbourhood of
$v_0$ is joined to $v_0$ by a unique geodesic, which crosses the cut surface
$\Sigma_{(0)}$ at $v_0$ exactly once (inside each cube at $v_0$ geodesics from $v_0$ are
straight segments, and the cut is an affine hyperplane). Mapping $y$ to this crossing
point defines a retraction $\rho$ onto $\Sigma_{(0)}\cong\Sigma_q$, continuous because
geodesics in CAT(0) spaces depend continuously on their endpoints. Define
$R\colon\widehat W\to X_0=X$ as follows:
\begin{itemize}
\item on the lift of $M_q$, $R=a\circ\rho$;
\item on $X_0$, $R$ is the identity, and on its cylinder $R(\sigma,t)=a(\sigma)$;
\item at $\lambda\ne0$: $\rho$ is defined on the whole closed cone neighbourhood of
$v_\lambda$, so its restriction to the cut surface $\Sigma_{(\lambda)}\cong\Sigma_q$
extends over a cone and is therefore homotopic to a constant map by a homotopy
$H\colon\Sigma_q\times[0,1]\to\Sigma_q$ with $H_1=\rho|_{\Sigma_{(\lambda)}}$ and $H_0$
constant. Put $R(\sigma,t)=a(H(\sigma,t))$ on the cylinder and let $R$ be constant on
$X_\lambda$.
\end{itemize}
These definitions agree on overlaps, so $R$ is a retraction of $\widehat W$ onto $X_0$.
Hence $\pi_1(X)\to\pi_1(\widehat W)$ is injective, and so is
$\pi_1(\widehat W)\to\pi_1(W)$; their composite is induced by $X\subset W$.

\emph{Step 2: a CAT(0) quotient (cf.\ \cite[Lemma~2.6]{DJW}).} Let $U\subset\widetilde W$
be the preimage of $X$. By Step~1 and Lemma~\ref{lem:components}, its components
$U_\alpha$ are copies of $\widetilde X$. Collapse each $U_\alpha$ to a point $v_\alpha$
to obtain a space $V$, and let $V_0=\{v_\alpha\}$. As $\widetilde W$ is simply connected
and the $U_\alpha$ are disjoint connected subcomplexes, $V\simeq\widetilde W\cup\bigcup_\alpha\Cone(U_\alpha)$
\cite[Ch.~0]{Hatcher} is simply connected by van Kampen's theorem. The cylinders
attached to $U_\alpha$ become cones with apex $v_\alpha$, and each truncated cube together
with the cone on its cut triangle becomes a full cube; so $V$ is a cube complex. The
vertex links at old vertices are copies of $L_q$. The link at $v_\alpha$ is the preimage
in $\widetilde W$ of the cut surface adjacent to $U_\alpha$, that is, the pullback of
$\widetilde X\to X$ along $a$: a disjoint union of covers of $L_q$. Covers of $L_q$ are
flag: the vertices of a clique in a cover project to pairwise adjacent, hence distinct,
vertices of $L_q$, which span a simplex; this simplex lifts to a simplex through any
vertex of the clique, and since a covering map is an isomorphism on the star of a vertex,
the lift contains the whole clique. Hence $V$ is CAT(0) (by \cite{Leary}, as $V$ need not
be locally finite), so contractible.

\emph{Step 3: homology.} Since $(\widetilde W,U)\to(V,V_0)$ is a relative homeomorphism
of CW pairs and $V_0$ is discrete,
$H_2(\widetilde W,U)\cong H_2(V,V_0)\cong H_2(V)=0$. Hence
$H_2(U)\to H_2(\widetilde W)$ is surjective. By Lemma~\ref{lem:components}, $\pi_1(W)$
permutes the components of $U$ transitively, and the component through the base point
is a copy of $\widetilde X$ with $H_2=\pi_2(X)$. Since $\pi_2(W)=H_2(\widetilde W)$, it
is generated over $\Z[\pi_1W]$ by the image of $\pi_2(X)$.
\end{proof}

Lemma~\ref{lem:geometric-pushout} gives~\eqref{eq:block-generation} for a single
replacement, and by induction (generation is preserved under composition, as in the proof
of Lemma~\ref{lem:structural}) for finitely many. For an arbitrary set of replacements,
let $E'$ be the full replacement and $s\in\pi_2(E')$. A representing sphere is compact,
so it meets only finitely many blocks, coming from a finite set $\Phi$ of relators. Let
$E'_\Phi$ be the presentation in which only the relators in $\Phi$ are replaced. The
structural map $g\colon E'_\Phi\to E'$ replacing the remaining relators is the identity
on the subcomplex containing the sphere, so $s=g_*(s')$ for some
$s'\in\pi_2(E'_\Phi)=\Z[G(E'_\Phi)]\cdot f_{\Phi*}\pi_2(E)$. Since $g\circ f_\Phi=f$ and
$g_*$ is equivariant, $s\in\Z[G(E')]\cdot f_*\pi_2(E)$. This proves \ref{B2}.

\subsection{The capped block: proof of (B3)}
\begin{lemma}\label{lem:capped-block}
Let $V_q$ be $\cx{B_q}$ with a $2$-cell attached along each $s_h$ and each $t_h$. Then
$V_q$ is Cockcroft. Consequently the capped presentation of \ref{B3} is Cockcroft.
\end{lemma}
\begin{proof}
By homotopy invariance, $C_q=M_q\cup_{\Sigma_q}\Cone(\Sigma_q)\simeq
\cx{B_q}\cup_{b_q}\Cone(\Sigma_q)$. With the one-vertex structure on $\Sigma_q$, the cone
consists of the cone point, an edge (the cone on the vertex), $2$-cells (the cones on
$s_h,t_h$) and one $3$-cell $e^3$ (the cone on the polygon). Contracting the edge
\cite[Prop.~0.17]{Hatcher} gives $C_q\simeq V_q\cup e^3$.

The cellular boundary of $e^3$ consists of the $2$-chain $b_{q\#}[\Sigma_q]$ on
$\cx{B_q}$, and of the capping $2$-cells with coefficients the exponent sums of $s_h$
and $t_h$ in $\omega_q$, which vanish. The chain $b_{q\#}[\Sigma_q]$ is a $2$-cycle of
$\cx{B_q}$ representing the image of $[\Sigma_q]$, which is zero by
Lemma~\ref{lem:spine}(iii) since $b_q$ corresponds to the inclusion
$\Sigma_q\subset M_q$. A null-homologous $2$-cycle in a two-complex is the zero chain,
so $\partial e^3=0$, and $H_2(V_q)\to H_2(V_q\cup e^3)$ is an isomorphism. As
$V_q\cup e^3\simeq C_q$ is aspherical, $V_q$ is Cockcroft by
Lemma~\ref{lem:toolbox}(d). Finally, eliminating each $s_h,t_h$ by its capping
$2$-cell is a homotopy equivalence from $V_q$ to the presentation complex of the capped
presentation $\pres{Z_q\mid\beta_{q,j}(1,\dots,1,Z_q)}$.
\end{proof}
This proves \ref{B3}, and completes the proofs of Theorem~\ref{thm:blocks},
Proposition~\ref{prop:relative} and Theorem~A. \qed



\end{document}